\documentclass[11pt]{article}%
\usepackage{amsxtra,amscd}
\usepackage{graphicx}
\usepackage{amsmath}
\usepackage{amsfonts}
\usepackage{amssymb}%
\usepackage[textwidth=6.0in,textheight=9.0in,centering]{geometry}
\usepackage[amsmath,hyperref,thmmarks]{ntheorem}
\usepackage{natbib}
\usepackage{times}
\usepackage{paralist}

\def\1{{1\mkern-7mu1}}

\newcommand\bquote{\begin{quote}}
\newcommand\equote{\end{quote}}
\newcommand\bsmall{\begin{small}}
\newcommand\esmall{\end{small}}
\newcommand\bfootnotesize{\begin{footnotesize}}
\newcommand\efootnotesize{\end{footnotesize}}
\renewcommand{\theenumi}{{\alph{enumi}}}

\renewcommand{\theenumii}{\roman{enumii}}

\let\cite=\citealt
\newcommand{\eb}[1]{{\itshape\bfseries#1}{\index{#1}}}
\renewcommand{\emph}{\eb}

\renewcommand{\Gamma}{\varGamma}
\renewcommand{\Pi}{\varPi}
\renewcommand{\Sigma}{\varSigma}

\DeclareMathOperator{\Ad}{Ad}

\DeclareMathOperator{\Aut}{Aut}

\DeclareMathOperator{\End}{End}

\DeclareMathOperator{\Gal}{Gal}
\DeclareMathOperator{\GL}{GL}

\DeclareMathOperator{\Hom}{Hom}
\DeclareMathOperator{\id}{id}

\DeclareMathOperator{\Ker}{Ker}
\DeclareMathOperator{\Lie}{Lie}

\newcommand{\Rep}{\mathsf{Rep}}%representations
\newcommand{\ad}{\mathrm{ad}}

\newcommand{\diag}{\mathrm{diag}}

\theoremnumbering{arabic}
\theoremheaderfont{\scshape}
\RequirePackage{latexsym}
\theoremseparator{.}
\newtheorem{X}{X}[section]
\theorembodyfont{\upshape}
\newtheorem{plain}[X]{}
\theorembodyfont{\slshape}
\newtheorem{E}[X]{}

\newtheorem{corollary}[X]{Corollary}
\newtheorem{lemma}[X]{Lemma}
\newtheorem{proposition}[X]{Proposition}
\newtheorem{theorem}[X]{Theorem}
\theorembodyfont{\upshape}
\newtheorem{definition}[X]{Definition}
\newtheorem{example}[X]{Example}
\newtheorem{exercise}[X]{Exercise}

\newtheorem{remark}[X]{Remark}

\theorembodyfont{\small}

\theoremstyle{nonumberplain}
\theorembodyfont{\normalsize}
\theoremsymbol{\ensuremath{_\Box}}
\RequirePackage{amssymb}
\newtheorem{proof}{Proof}

\qedsymbol{\ensuremath{_\Box}}
\theoremclass{LaTeX}
\begin{document}

\title{Semisimple Algebraic Groups in Characteristic Zero
\footnote{\copyright 2007 J.S. Milne}}
\author{J.S. Milne}
\date{May 9, 2007}
\maketitle

\begin{abstract}
It is shown that the classification theorems for semisimple algebraic groups
in characteristic zero can be derived quite simply and naturally from the
corresponding theorems for Lie algebras by using a little of the theory of
tensor categories. This article is extracted from \cite{milneSSS}.

\end{abstract}

\section*{Introduction}

The classical approach to classifying the semisimple algebraic groups over
$\mathbb{C}{}$ (see \cite{borel1975}, \S 1) is to:

\begin{itemize}
\item classify the complex semisimple Lie algebras in terms of reduced root
systems (Killing, E.~Cartan, et al.);

\item classify the complex semisimple Lie groups with a fixed Lie algebra in
terms of certain lattices attached to the root system of the Lie algebra
(Weyl, E.~Cartan, et al.);

\item show that a complex semisimple Lie group has a unique structure of an
algebraic group compatible with its complex structure.
\end{itemize}

\noindent Chevalley (1956-58\nocite{chevalley1956-58},
1960-61\nocite{chevalley1960-61}) proved that the classification one obtains
is valid in all characteristics, but his proof is long and
complicated.\footnote{See \cite{humphreys1975}, Chapter XI, and
\cite{springer1998}, Chapters 10 \& 11. Despite its fundamental importance,
many books on algebraic groups, e.g., \cite{borel1991}, don't prove the
classification, and some, e.g., \cite{tauvelY2005}, don't even state it.}

Here I show that the classification theorems for semisimple algebraic groups
in characteristic zero can be derived quite simply and naturally from the
corresponding theorems for Lie algebras by using a little of the theory of
tensor categories. In passing, one also obtains a classification of their
finite-dimensional representations. Beyond its simplicity, the advantage of
this approach is that it makes clear the relation between semisimple Lie
algebras, semisimple algebraic groups, and tensor categories in characteristic zero.

The idea of obtaining an algebraic proof of the classification theorems for
semisimple algebraic groups in characteristic zero by exploiting their
representations is not new --- in a somewhat primitive form it can be found
already in Cartier's announcement (1956)\nocite{cartier1956} --- but I have
not seen an exposition of it in the literature.

Throughout, $k$ is a field of characteristic zero and \textquotedblleft
representation\textquotedblright\ of a Lie algebra or affine group means
\textquotedblleft finite-dimensional linear representation\textquotedblright.

I assume that the reader is familiar with the elementary parts of the theories
of algebraic groups and tensor categories and with the classification of
semisimple Lie algebras; see \cite{milneSSS} for a more detailed account.

\section{Elementary Tannaka duality}

\begin{E}
\label{s25}Let $G$ be an algebraic group, and let $R$ be a $k$-algebra.
Suppose that for each representation $(V,r_{V})$ of $G$ on a
finite-dimensional $k$-vector space $V$, we have an $R$-linear endomorphism
$\lambda_{V}$ of $V(R)$. If the family $(\lambda_{V})$ satisfies the conditions,

\begin{itemize}
\item $\lambda_{V\otimes W}=\lambda_{V}\otimes\lambda_{W}$ for all
representations $V,W$,

\item $\lambda_{\1}=\id_{\1}$ (here $\1=k$ with the trivial action),

\item $\lambda_{W}\circ\alpha_{R}=\alpha_{R}\circ\lambda_{V}$, for all
$G$-equivariant maps $\alpha\colon V\rightarrow W,$\noindent
\end{itemize}

\noindent then there exists a $g\in G(R)$ such that $\lambda_{V}=r_{V}(g)$ for
all $X$ (\cite{deligneM1982}, 2.8.
\end{E}

\noindent Because $G$ admits a faithful finite-dimensional representation, $g$ is uniquely determined by the family $(\lambda_{V})$, and so
the map sending $g\in G(R)$ to the family $(r_{V}(g))$ is a bijection from
$G(R)$ onto the set of families satisfying the conditions in the theorem.
Therefore we can recover $G$ from the category $\Rep(G)$ of representations of
$G$ on finite-dimensional $k$-vector spaces.\medskip

\begin{plain}
Let $G$ be an algebraic group over $k$. For each $k$-algebra $R$, let
$G^{\prime}(R)$ be the set of families $(\lambda_{V})$ satisfying the
conditions in (\ref{s25}). Then $G^{\prime}$ is a functor from $k$-algebras to
groups, and there is a natural map $G\rightarrow G^{\prime}$. That this map is
an isomorphism is often paraphrased by saying that \emph{Tannaka duality holds
for} $G$.
\end{plain}

\section{Gradations on tensor categories}

\begin{plain}
\label{s54}Let $M$ be a finitely generated abelian group, and let $D(M)$ be
the associated diagonalizable algebraic group. An $M$-\emph{gradation} on an
object $X$ of an abelian category is a family of subobjects $(X^{m})_{m\in M}$
such that $X=\bigoplus_{m\in M}X^{m}$. An $M$-\emph{gradation} on a tensor
category $\mathsf{C}$ is an $M$-gradation on each object $X$ of $\mathsf{C}$
compatible with all arrows in $\mathsf{C}$ and with tensor products in the
sense that $(X\otimes Y)^{m}=\bigoplus_{r+s=m}X^{r}\otimes X^{s}$. Let
$(\mathsf{C},\omega)$ be a neutral tannakian category, and let $G\ $be its
Tannaka dual. To give an $M$-gradation on $\mathsf{C}$ is the same as to give
a central homomorphism $D(M)\rightarrow G$: a homomorphism corresponds to the
$M$-gradation such that $X^{m}$ is the subobject of $X$ on which $D(M)$ acts
through the character $m$ (\cite{saavedra1972}; \cite{deligneM1982}, \S 5).
\end{plain}

\begin{plain}
\label{s55}Let $\mathsf{C}$ be a semsimple $k$-linear tensor category such
that $\End(X)=k$ for every simple object $X$ in $\mathsf{C}$, and let
$I(\mathsf{C})$ be the set of isomorphism classes of simple objects in
$\mathsf{C}$. For elements $x,x_{1},\ldots,x_{m}$ of $I(\mathsf{C})$
represented by simple objects $X,X_{1},\ldots,X_{m}$, write $x\prec
x_{1}\otimes\cdots\otimes x_{m}$ if $X\ $is a direct factor of $X_{1}%
\otimes\cdots\otimes X_{m}$. The following statements are obvious.

\begin{enumerate}
\item Let $M$ be a commutative group. To give an $M$-gradation on $\mathsf{C}$
is the same as to give a map $f\colon I(\mathsf{C})\rightarrow M$ such that
\[
x\prec x_{1}\otimes x_{2}\implies f(x)=f(x_{1})+f(x_{2}).
\]
A map from $I(\mathsf{C})$ to a commutative group satisfying this condition
will be called a \emph{tensor map}. For such a map, $f(\1)=0$, and if $X$ has
dual $X^{\vee}$, then $f([X^{\vee}])=-f([X])$.

\item Let $M(\mathsf{C})$ be the free abelian group with generators the
elements of $I(\mathsf{C})$ modulo the relations: $x=x_{1}+x_{2}$ if $x\prec
x_{1}\otimes x_{2}$. The obvious map $I(\mathsf{C})\rightarrow M(\mathsf{C})$
is a universal tensor map, i.e., it is a tensor map, and every other tensor
map $I(\mathsf{C})\rightarrow M$ factors uniquely through it. Note that
$I(\mathsf{C})\rightarrow M(\mathsf{C})$ is surjective.
\end{enumerate}
\end{plain}

\begin{plain}
\label{s56}Let $(\mathsf{C},\omega)$ be a neutral tannakian category such that
$\mathsf{C}$ is semisimple and $\End(V)=k$ for every simple object in
$\mathsf{C}$. Let $Z$ be the centre of $G\overset{\text{{\tiny def}}}%
{=}\underline{\Aut}^{\otimes}(\omega)$. Because $\mathsf{C}$ is semisimple,
$G$ is reductive, and so $Z$ is of multiplicative type. Assume (for
simplicity) that $Z$ is split, so that $Z=D(N)$ with $N$ the group of
characters of $Z$. According to (\ref{s54}), to give an $M$-gradation on
$\mathsf{C}$ is the same as to give a homomorphism $D(M)\rightarrow Z$, or,
equivalently, a homomorphism $N\rightarrow M$. On the other hand, (\ref{s55})
shows that to give an $M$-gradation on $\mathsf{C}$ is the same as to give a
homomorphism $M(\mathsf{C})\rightarrow M$. Therefore $M(\mathsf{C})\simeq N$.
In more detail: let $X$ be an object of $\mathsf{C}$; if $X$ is simple, then
$Z$ acts on $X$ through a character $n$ of $Z$, and the tensor map $[X]\mapsto
n\colon$ $I(\mathsf{C})\rightarrow N$ is universal.
\end{plain}

\begin{plain}
\label{s56a}Let $(\mathsf{C},\omega)$ be as in (\ref{s56}), and define an
equivalence relation on $I(\mathsf{C})$ by%
\[
a\sim a^{\prime}\iff\text{there exist }x_{1},\ldots,x_{m}\in I(\mathsf{C}%
)\text{ such that }a,a^{\prime}\prec x_{1}\otimes\cdots\otimes x_{m}.
\]
A function $f$ from $I(\mathsf{C})$ to a commutative group defines a gradation
on $\mathsf{C}$ if and only if $f(a)=f(a^{\prime})$ whenever $a\sim a^{\prime
}$. Therefore, $M(\mathsf{C})\simeq I(\mathsf{C})/\!\!\sim$ .
\end{plain}

\section{Representations of split semisimple Lie algebras}

Throughout this subsection, $(\mathfrak{g},\mathfrak{h}{})$ is a split
semisimple Lie algebra with root system $R\subset\mathfrak{h}{}^{\vee}$, and
$\mathfrak{b}{}$ is the Borel subalgebra of $(\mathfrak{g}{},\mathfrak{h}{})$
attached to a base $S$ for $R$. According to a theorem of Weyl, the
representations of $\mathfrak{g}{}$ are semisimple, and so to classify them it
suffices to classify the simple representations.

\begin{E}
\label{s103}Let $r\colon\mathfrak{g}\rightarrow\mathfrak{g}\mathfrak{l}_{V}$
be a simple representation of $\mathfrak{g}{}$.

\begin{enumerate}
\item There exists a unique one-dimensional subspace $L$ of $V$ stabilized by
$\mathfrak{b}{}$.

\item The $L$ in (a) is a weight space for $\mathfrak{h}{}$, i.e.,
$L=V_{\varpi_{V}}$ for some $\varpi_{V}\in\mathfrak{h}{}^{\vee}$.

\item The $\varpi_{V}$ in (b) is dominant, i.e., $\varpi_{V}\in P_{++}$;

\item If $\varpi$ is also a weight for $\mathfrak{h}{}$ in $V$, then
$\varpi=\varpi_{V}-\sum_{\alpha\in S}m_{\alpha}\alpha$ with $m_{\alpha}%
\in\mathbb{N}{}$.
\end{enumerate}
\end{E}

\noindent The Lie-Kolchin theorem shows that there does exist a
one-dimensional eigenspace for $\mathfrak{b}{}$ --- the content of (a) is that
when $V$ is simple (as a representation of $\mathfrak{g}{}$), the space is
unique. Since $L$ is mapped into itself by $\mathfrak{b}{}$, it is also mapped
into itself by $\mathfrak{h}{}$, and so lies in a weight space. The content of
(b) is that it is the whole weight space. For the proof, see
\cite{bourbakiLie}, VIII, \S 7.

Because of (d), $\varpi_{V}$ is called the \emph{heighest weight} of the
simple representation $(V,r)$.

\begin{E}
\label{s104}Every dominant weight occurs as the highest weight of a simple
representation of $\mathfrak{g}{}$ (ibid.).
\end{E}

\begin{E}
\noindent\label{s104a}Two simple representations of $\mathfrak{g}{}$ are
isomorphic if and only if their highest weights are equal.
\end{E}

\noindent Thus $(V,r)\mapsto\varpi_{V}$ defines a bijection from the set of
isomorphism classes of simple representations of $\mathfrak{g}{}$ onto the set
of dominant weights $P_{++}$.

\begin{E}
\label{s105}If $(V,r)$ is a simple representation of $\mathfrak{g}{}$, then
$\End(V,r)\simeq k$.
\end{E}

\noindent To see this, let $V=V_{\varpi}$ with $\varpi$ dominant. Every
isomorphism $V_{\varpi}\rightarrow V_{\varpi}$ maps the highest weight line
$L$ into itself, and is determined by its restriction to $L$ because $L$
generates $V_{\varpi}$ as a $\mathfrak{g}{}$-module.

\begin{plain}
\label{s106}The category $\Rep(\mathfrak{g}{})$ of representations of
$\mathfrak{g}$ is a semisimple $k$-linear category to which we can apply
(\ref{s55}). Statements (\ref{s104}, \ref{s104a}) allow us to identify the set
of isomorphism classes of $\Rep(\mathfrak{g}{})$ with $P_{++}$. Let
$M(P_{++})$ be the free abelian group with generators the elements of $P_{++}$
and relations%
\[
\varpi=\varpi_{1}+\varpi_{1}\text{ if }V_{\varpi}\subset V_{\varpi_{1}}\otimes
V_{\varpi_{2}}.
\]
Then $P_{++}\rightarrow M(P_{++})$ is surjective, and two elements $\varpi$
and $\varpi^{\prime}$ of $P_{++}$ have the same image in $M(P_{++})$ if and
only there exist $\varpi_{1},\ldots,\varpi_{m}\in P_{++}$ such that
$W_{\varpi}$ and $W_{\varpi^{\prime}}$ are subrepresentations of
$W_{\varpi_{1}}\otimes\cdots\otimes W_{\varpi_{m}}$ (see \ref{s56a}). Later we
shall prove that this condition is equivalent to $\varpi-\varpi^{\prime}\in
Q$, and so $M(P_{++})\simeq P/Q$. In other words, $\Rep(\mathfrak{g}{})$ has a
gradation by $P_{++}/Q\cap P_{++}\simeq P/Q$ but not by any larger quotient.

For example, let $\mathfrak{g}{}=\mathfrak{sl}_{2}$, so that $Q=\mathbb{Z}%
{}\alpha$ and $P=\mathbb{Z}{}\frac{\alpha}{2}$. For $n\in\mathbb{N}{}$, let
$V(n)$ be a simple representation of $\mathfrak{g}{}$ with heighest weight
$\frac{n}{2}\alpha$. From the Clebsch-Gordon formula (\cite{bourbakiLie},
VIII, \S 9), namely,%
\[
V(m)\otimes V(n)\approx V(m+n)\oplus V(m+n-2)\oplus\cdots\oplus V(m-n),\quad
n\leq m,
\]
we see that $\Rep(\mathfrak{g}{})$ has a natural $P/Q$-gradation (but not a
gradation by any larger quotient of $P$).
\end{plain}

\begin{exercise}
\label{s107}\footnote{Not done by the author.}Prove that the kernel of
$P_{++}\rightarrow M(P_{++})$ is $Q\cap P_{++}$ by using the formulas for the
characters and multiplicities of the tensor products of simple representations
(cf. \cite{humphreys1972}, \S 24, especially Exercise 12).
\end{exercise}

\section{Basic theory of semisimple algebraic groups}

\begin{proposition}
\label{s110}A connected algebraic group $G$ is semisimple (resp. reductive) if
and only if its Lie algebra is semisimple (resp. reductive).
\end{proposition}

\begin{proof}
Suppose that $\Lie(G)$ is semisimple, and let $N$ be a normal commutative
subgroup of $G$. Then $\Lie(N)$ is a commutative ideal in $\Lie(G)$, and so is
zero. This implies that $N\ $is finite.

Conversely, suppose that $G$ is semisimple, and let $\mathfrak{n}{}$ be a
commutative ideal in $\mathfrak{g}{}$. When $G$ acts on $\mathfrak{g}{}$
through the adjoint representation, the Lie algebra of $H\overset
{\text{{\tiny def}}}{=}C_{G}(\mathfrak{n}{})$ is%
\[
\mathfrak{h}{}=\{x\in\mathfrak{g}{}\mid\lbrack x,\mathfrak{n}{}]=0\},
\]
which contains $\mathfrak{n}{}$. Because $\mathfrak{n}{}$ is an ideal, so is
$\mathfrak{h}{}$:%
\[
\lbrack x,n]=0,\quad y\in\mathfrak{g}{}\implies\lbrack\lbrack
y,x],n]=[y,[x,n]]-[x,[y,n]]=0.
\]
Therefore $H^{\circ}$ is normal in $G$, which implies that its centre
$Z(H^{\circ})$ is normal in $G$. Because $G$ is semisimple, $Z(H^{\circ})$ is
finite, and so $z(\mathfrak{h}{})=0$. But $z(\mathfrak{h}{})\supset
\mathfrak{n}{}$, and so $\mathfrak{n}{}=0$.

The reductive case is similar.
\end{proof}

\begin{corollary}
\label{s111}The Lie algebra of the radical of a connected algebraic group $G$
is the radical of the Lie algebra of $\mathfrak{g}{}$; in other words,
$\Lie(R(G))=r(\Lie(G))$.
\end{corollary}

\begin{proof}
Because $\Lie$ is an exact functor, the exact sequence%
\[
1\rightarrow RG\rightarrow G\rightarrow G/RG\rightarrow1
\]
gives rise to an exact sequence%
\[
0\rightarrow\Lie(RG)\rightarrow\mathfrak{g}{}\rightarrow\Lie(G/RG)\rightarrow
0
\]
in which $\Lie(RG)$ is solvable (obviously) and $\Lie(G/RG)$ is semisimple.
The image in $\Lie(G/RG)$ of any solvable ideal in $\mathfrak{g}{}$ is zero,
and so $\Lie(RG)$ is the largest solvable ideal in $\mathfrak{g}{}$.
\end{proof}

A connected algebraic group $G$ is \emph{simple} if it is noncommutative and
has no proper normal algebraic subgroups $\neq1$, and it is \emph{almost
simple} if it is noncommutative and has no proper normal algebraic subgroups
except for finite subgroups. An algebraic group $G$ is said to be the
\emph{almost-direct product} of its algebraic subgroups $G_{1},\ldots,G_{n}$
if the map%
\[
(g_{1},\ldots,g_{n})\mapsto g_{1}\cdots g_{n}\colon G_{1}\times\cdots\times
G_{n}\rightarrow G
\]
is a surjective homomorphism with finite kernel; in particular, this means
that the $G_{i}$ commute with each other and each $G_{i}$ is normal in $G$.

\begin{theorem}
\label{s112}Every connected semisimple algebraic group $G$ is an almost-direct
product
\[
G_{1}\times\cdots\times G_{r}\rightarrow G
\]
of its minimal connected normal algebraic subgroups. In particular, there are
only finitely many such subgroups. Every connected normal algebraic subgroup
of $G$ is a product of those $G_{i}$ that it contains, and is centralized by
the remaining ones.
\end{theorem}

\begin{proof}
Because $\Lie(G)$ is semisimple, it is a direct sum of its simple ideals:%
\[
\Lie(G)=\mathfrak{g}{}_{1}\oplus\cdots\oplus\mathfrak{g}{}_{r}.
\]
Let $G_{1}$ be the identity component of $C_{G}(\mathfrak{g}{}_{2}\oplus
\cdots\oplus\mathfrak{g}{}_{r})$. Then
\[
\Lie(G_{1})=c_{\mathfrak{g}{}}(\mathfrak{g}{}_{2}\oplus\cdots\oplus
\mathfrak{g}{}_{r})=\mathfrak{g}{}_{1},
\]
which is an ideal in $\Lie(G)$, and so $G_{1}$ is normal in $G$. If $G_{1}$
had a proper normal nonfinite algebraic subgroup, then $\mathfrak{g}{}_{1}%
$would have an ideal other than $\mathfrak{g}{}_{1}$ and $0$, contradicting
its simplicity. Therefore $G_{1}$ is almost-simple. Construct $G_{2}%
,\ldots,G_{r}$ similarly. Because $[\mathfrak{g}{}_{i},\mathfrak{g}{}_{j}]=0$,
the groups $G_{i}$ and $G_{j}$ commute. The subgroup $G_{1}\cdots G_{r}$ of
$G$ has Lie algebra $\mathfrak{g}{}$, and so equals $G$. Finally,
\[
\Lie(G_{1}\cap\ldots\cap G_{r})=\mathfrak{g}{}_{1}\cap\ldots\cap\mathfrak{g}%
{}_{r}=0
\]
and so $G_{1}\cap\ldots\cap G_{r}$ is finite.

Let $H$ be a connected algebraic subgroup of $G$. If $H$ is normal, then $\Lie
H$ is an ideal, and so it is a direct sum of those $\mathfrak{g}{}_{i}$ it
contains and centralizes the remainder. This implies that $H$ is a product of
those $G_{i}$ it contains, and centralizes the remainder.
\end{proof}

\begin{corollary}
\label{s113} An algebraic group is semisimple if and only if it is an almost
direct product of almost-simple algebraic groups.
\end{corollary}

\begin{corollary}
\label{s114}All nontrivial connected normal subgroups and quotients of a
semisimple algebraic group are semisimple.
\end{corollary}

\begin{proof}
They are almost-direct products of almost-simple algebraic groups.
\end{proof}

\begin{corollary}
\label{s115} A semisimple group has no commutative quotients $\neq1$.
\end{corollary}

\begin{proof}
\noindent This is obvious for simple groups, and the theorem then implies it
for semisimple groups.
\end{proof}

\begin{definition}
\label{s115a}A semisimple algebraic group $G$ is said to be \emph{splittable}
if it has a split maximal subtorus. A \emph{split semisimple algebraic group}
is a pair $(G,T)$ consisting of a semisimple algebraic group $G$ and a split
maximal torus $T$.
\end{definition}

\begin{lemma}
\label{s115b}If $T$ is a split torus in $G$, then $\Lie(T)$ is a commutative
subalgebra of $\Lie(G)$ consisting of semisimple elements.
\end{lemma}

\begin{proof}
Certainly $\Lie(T)$ is commutative. Let $(V,r_{V})$ be a faithful
representation of $G$. Then $(V,r_{V})$ decomposes into a direct sum
$\bigoplus\nolimits_{\chi\in X^{\ast}(T)}V_{\chi}$, and $\Lie(T)$ acts
(semisimply) on each factor $V_{\chi}$ through the character $d\chi$. As
$(V,dr_{V})$ is faithful, this shows that $\Lie(T)$ consists of semisimple elements.
\end{proof}

\section{Rings of representations of Lie algebras}

Let $\mathfrak{g}{}$ be a Lie algebra over $k$. A \emph{ring of
representations} of $\mathfrak{g}{}$ is a collection of representations of
$\mathfrak{g}{}$ that is closed under the formation of direct sums,
subquotients, tensor products, and duals. An \emph{endomorphism} of such a
ring $\mathcal{R}{}$ is a family
\[
\alpha=(\alpha_{V})_{V\in\mathcal{R}{}},\quad\alpha_{V}\in
\End_{k\text{-linear}}(V),
\]
such that

\begin{itemize}
\item $\alpha_{V\otimes W}=\alpha_{V}\otimes\id_{W}+\id_{V}\otimes\alpha_{W}$
for all $V,W\in\mathcal{R}{}$,

\item $\alpha_{V}=0$ if $\mathfrak{g}{}$ acts trivially on $V$, and

\item for any homomorphism $\beta\colon V\rightarrow W$ of representations in
$\mathcal{R}{}{}$,%
\[
\alpha_{W}\circ\beta=\alpha_{V}\circ\beta.
\]

\end{itemize}

\noindent The set $\mathfrak{g}{}_{R}$ of all endomorphisms of $\mathcal{R}{}$
becomes a Lie algebra over $k$ (possibly infinite dimensional) with the
bracket%
\[
\lbrack\alpha,\beta]_{V}=[\alpha_{V},\beta_{V}].
\]

\begin{example}
[\cite{iwahori1954}]\label{s116}Let $\mathfrak{g}{}=k$ with $k$ algebraically
closed. To give a represention of $\mathfrak{g}{}$ on a vector space $V$ is
the same as to give an endomorphism $\alpha$ of $V$, and so the category of
representations of $\mathfrak{g}{}{}$ is equivalent to the category of pairs
$(k^{n},A),$ $n\in\mathbb{N}{}$, with $A$ an $n\times n$ matrix. It follows
that to give an endomorphism of the ring $\mathcal{R}{}$ of all
representations of $\mathfrak{g}{}$ is the same as to give a map
$A\mapsto\lambda(A)$ sending a square matrix $A$ to a matrix of the same size
and satisfying certain conditions. A pair $(g,c)$ consisting of an additive
homomorphism $g\colon k\rightarrow k$ and an element $c$ of $k$ defines a
$\lambda$ as follows:

\begin{itemize}
\item $\lambda(S)=U\diag(ga_{1},\ldots,ga_{n})U^{-1}$ if $\lambda$ is the
semisimple matrix $U\diag(a_{1},\ldots,a_{n})U^{-1}$;

\item $\lambda(N)=cN$ if $N$ is nilpotent;

\item $\lambda(A)=\lambda(S)+\lambda(N)$ if $A=S+N$ is the decomposition of
$A$ into its commuting semisimple and nilpotent parts.
\end{itemize}

\noindent Moreover, every $\lambda$ arises from a unique pair $(g,c)$. Note
that $\mathfrak{g}{}_{\mathcal{R}{}}$ has infinite dimension.
\end{example}

Let $\mathcal{R}{}$ be a ring of representations of a Lie algebra
$\mathfrak{g}{}$. For any $x\in\mathfrak{g}{}$, $(r_{V}(x))_{V\in\mathcal{R}%
{}}$ is an endomorphism of $\mathcal{R}{}$, and $x\mapsto(r_{V}(x))$ is a
homomorphism of Lie algebras $\mathfrak{g}{}\rightarrow\mathfrak{g}%
{}_{\mathcal{R}{}}$.

\begin{lemma}
\label{s117}If $\mathcal{R}{}$ contains a faithful representation of
$\mathfrak{g}{}$, then $\mathfrak{g}{}\rightarrow\mathfrak{g}{}_{\mathcal{R}%
{}}$ is injective.
\end{lemma}

\begin{proof}
For any representation $(V,r_{V})$ of $\mathfrak{g}{}$, the composite%
\[
\mathfrak{g}{}\xrightarrow{x\mapsto(r_{V}(x))}\mathfrak{g}{}_{\mathcal{R}%
}\xrightarrow{\lambda\mapsto\lambda_{V}}\mathfrak{g}{}\mathfrak{l}{}(V).
\]
is $r_{V}$. Therefore, $\mathfrak{g}{}\rightarrow\mathfrak{g}{}_{\mathcal{R}%
{}}$ is injective if $r_{V}$.
\end{proof}

\begin{proposition}
\label{s118}Let $G$ be an affine group over $k$, and let $\mathcal{R}{}$ be
the ring of representations of $\mathfrak{g}{}$ arising from a representation
of $G$. Then $\mathfrak{g}{}_{\mathcal{R}{}}\simeq\Lie(G)$; in particular,
$\mathfrak{g}{}_{\mathcal{R}{}}$ depends only of $G^{\circ}$.
\end{proposition}

\begin{proof}
By definition, $\Lie(G)$ is the kernel of $G(k[\varepsilon])\rightarrow G(k)$.
Therefore, to give an element of $\Lie(G)$ is the same as to give a family of
$k[\varepsilon]$-linear maps
\[
\id_{V}+\alpha_{V}\varepsilon\colon V[\varepsilon]\rightarrow V[\varepsilon]
\]
indexed by $V\in\mathcal{R}{}$ satisfying the three conditions of (\ref{s25}).
The first of these conditions says that%
\[
\id_{V\otimes W}+\alpha_{V\otimes W}\varepsilon=(\id_{V}+\alpha_{V}%
\varepsilon)\otimes(\id_{W}+\alpha_{W}\varepsilon),
\]
i.e., that%
\[
\alpha_{V\otimes W}=\id_{V}\otimes\alpha_{W}+\alpha_{V}\otimes\id_{W}.
\]
The second condition says that%
\[
\alpha_{\1}=0,
\]
and the third says that the $\alpha_{V}$ commute with all $G$-morphisms
($=\mathfrak{g}{}$-morphisms). Therefore, to give such a family is the same as
to give an element $(\alpha_{V})_{V\in\mathcal{R}{}}$ of
$\mathcal{\mathfrak{g}{}}_{\mathcal{R}{}}$.
\end{proof}

\begin{proposition}
\label{s118a}For a ring $\mathcal{R}{}$ of representations of a Lie algebra
$\mathfrak{g}{}$, the following statements are equivalent:

\begin{enumerate}
\item the map $\mathfrak{g}{}\rightarrow\mathfrak{g}{}_{\mathcal{R}{}}$ is an isomorphism;

\item $\mathfrak{g}{}$ is the Lie algebra of an affine group $G$ such that
$G^{\circ}$ is algebraic and $\mathcal{R}{}$ is the ring of all
representations of $\mathfrak{g}{}$ arising from a representation of $G$.
\end{enumerate}
\end{proposition}

\begin{proof}
This is an immediate consequence of (\ref{s118}) and the fact that an affine
group is algebraic if its Lie algebra is finite-dimensional.
\end{proof}

\begin{corollary}
\label{s118d}Let $\mathfrak{g}{}\rightarrow\mathfrak{g}{}\mathfrak{l}{}(V)$ be
a faithful representation of $\mathfrak{g}{}$, and let $\mathcal{R}(V){}$ be
the ring of representations of $\mathfrak{g}{}$ generated by $V$. Then
$\mathfrak{g}{}\rightarrow\mathfrak{g}{}_{\mathcal{R}{}(V)}$ is an isomorphism
if and only if $\mathfrak{g}{}$ is algebraic, i.e., the Lie algebra of an
algebraic subgroup of $\GL_{V}$.
\end{corollary}

\begin{proof}
Immediate consequence of the proposition.
\end{proof}

\begin{remark}
\label{s118b} Let $\mathfrak{g}{}\rightarrow\mathfrak{g}{}\mathfrak{l}{}(V)$
be a faithful representation of $\mathfrak{g}{}$, and let $\mathcal{R}(V){}$
be the ring of representations of $\mathfrak{g}{}$ generated by $V$. When is
$\mathfrak{g}{}\rightarrow\mathfrak{g}{}_{\mathcal{R}(V){}}$ an isomorphism?
It is easy to show, for example, when $\mathfrak{g}{}=[\mathfrak{g}%
{},\mathfrak{g}{}]$. In particular, $\mathfrak{g}{}\rightarrow\mathfrak{g}%
{}_{\mathcal{R}(V){}}$ is an isomorphism when $\mathfrak{g}{}$ is semisimple.
For an abelian Lie group $\mathfrak{g}{}$, $\mathfrak{g}{}\rightarrow
\mathfrak{g}{}_{\mathcal{R}(V){}}$ is an isomorphism if and only if
$\mathfrak{g}{}\rightarrow\mathfrak{g}{}\mathfrak{l}{}(V)$ is a semisimple
representation and there exists a lattice in $\mathfrak{g}{}$ on which the
characters of $\mathfrak{g}{}$ in $V$ take integer values. For the Lie algebra
in (\cite{bourbakiLie}, I, \S 5, Exercise 6), $\mathfrak{g}{}\rightarrow
\mathfrak{g}{}_{\mathcal{R}(V){}}$ is \textit{never} an isomorphism.
\end{remark}

\label{s118c}Let $\mathcal{R}{}$ be the ring of all representations of
$\mathfrak{g}{}$. When $\mathfrak{g}{}\rightarrow\mathfrak{g}{}_{\mathcal{R}%
{}}$ is an isomorphism one says that \emph{Tannaka duality holds for}
$\mathfrak{g}{}$. The aside shows that Tannaka duality holds for
$\mathfrak{g}{}$ if $[\mathfrak{g}{},\mathfrak{g}{}]=\mathfrak{g}{}$. On the
other hand, Example \ref{s116} shows that Tannaka duality fails when
$[\mathfrak{g}{},\mathfrak{g}{}]\neq\mathfrak{g}$, and even that
$\mathfrak{g}{}_{\mathcal{R}{}}$ has infinite dimension in this case.

\section{An adjoint to the functor $\Lie$}

Let $\mathfrak{g}{}$ be a Lie group, and let $\mathcal{R}{}$ be the ring of
all representations of $\mathfrak{g}{}$ . We define $G(\mathfrak{g}{})$ to be
the Tannaka dual of the neutral tannakian category $\Rep(\mathfrak{g}{})$.
Recall that this means that $G(\mathfrak{g}{})$ is the affine group whose
$R$-points for any $k$-algebra $R$ are the families%
\[
\lambda=(\lambda_{V})_{V\in\mathcal{R}{}},\quad\lambda_{V}\in
\End_{R\text{-linear}}(V(R)),
\]
such that

\begin{itemize}
\item $\lambda_{V\otimes W}=\lambda_{V}\otimes\lambda_{W}$ for all
$V\in\mathcal{R}{};$

\item if $xv=0$ for all $x\in\mathfrak{g}{}$ and $v\in V$, then $\lambda
_{V}v=v$ for all $\lambda\in$$G(\mathfrak{g}{})(R)$ and $v\in V(R)$;

\item for every $\mathfrak{g}{}$-homomorphism $\beta\colon V\rightarrow W$,%
\[
\lambda_{W}\circ\beta=\beta\circ\lambda_{V}.
\]

\end{itemize}

\noindent For each $V\in\mathcal{R}{}$, there is a representation $r_{V}$ of
$G(\mathfrak{g}{})$ on $V$ defined by%
\[
r_{V}(\lambda)v=\lambda_{V}v,\quad\lambda\in G(\mathfrak{g}{})(R),\quad v\in
V(R),\quad R\text{ a }k\text{-algebra,}%
\]
and $V\rightsquigarrow(V,r_{V})$ is an equivalence of categories%
\begin{equation}
\Rep(\mathfrak{g}{})\overset{\sim}{\longrightarrow}\Rep(G(\mathfrak{g}{})).
\label{e12}%
\end{equation}

\begin{lemma}
\label{s118e}The homomorphism $\eta\colon\mathfrak{g}{}\rightarrow
\Lie(G(\mathfrak{g}{}))$ is injective, and the composite of the functors%
\begin{equation}
\Rep(G{}(\mathfrak{g}{}%
))\xrightarrow{(V,r)\rightsquigarrow (V,dr)}\Rep(\Lie(G{}(\mathfrak{g}%
{})))\xrightarrow{\eta^{\vee}}\Rep(\mathfrak{g}{}) \label{e11}%
\end{equation}
is an equivalence of categories.
\end{lemma}

\begin{proof}
According to (\ref{s118}), $\Lie($$G(\mathfrak{g}{}))\simeq\mathfrak{g}%
_{\mathcal{R}{}}$, and so the first assertion follows from (\ref{s117}) and
Ado's theorem. The composite of the functors in (\ref{e11}) is a quasi-inverse
to the functor in (\ref{e12}).
\end{proof}

\begin{lemma}
\label{s118f}The affine group $G(\mathfrak{g}{})$ is connected.
\end{lemma}

\begin{proof}
We have to show that if a representation $V$ of $\mathfrak{g}{}$ has the
property that the category of subquotients of direct sums of copies of
$V\mathfrak{\ }$is stable under tensor products, then $V$ is a trivial
representation. When $\mathfrak{g}{}=k$, this is obvious (cf. \ref{s116}), and
when $\mathfrak{g}{}$ is semisimple it follows from (\ref{s103}).

Let $V$ be a representation of $\mathfrak{g}{}$ with the property. It follows
from the commutative case that the radical of $\mathfrak{g}{}$ acts trivially
on $V$, and then it follows from the semisimple case that $\mathfrak{g}{}$
itself acts trivially.
\end{proof}

\begin{proposition}
\label{s119}The pair $($$G(\mathfrak{g}{}),\eta)$ is universal: for any
algebraic group $H$ and $k$-algebra homomorphism $a\colon\mathfrak{g}%
\rightarrow\Lie(H)$, there is a unique homomorphism $b\colon$$G(\mathfrak{g}%
{})\rightarrow H$ such that $a=\Lie(b)\circ\eta$. In other words, the map
sending a homomorphism $b\colon G(\mathfrak{g}{})\rightarrow H$ to the
homomorphism $\Lie(b)\circ\eta\colon\mathfrak{g}{}\rightarrow\Lie(H)$ is a
bijection%
\begin{equation}
\Hom_{\text{affine groups}}(G{}(\mathfrak{g}{}),H)\rightarrow\Hom_{\text{Lie
algebras}}(\mathfrak{g}{},\Lie(H)). \label{e9}%
\end{equation}
If $a$ is surjective and $\Rep($$G(\mathfrak{g}{}))$ is semisimple, then $b$
is surjective.
\end{proposition}

\begin{proof}
From a homomorphism $b\colon$$G(\mathfrak{g}{})\rightarrow H$, we get a
commutative diagram%
\[
\begin{CD}
\Rep(H) @>{b^{\vee}}>> \Rep(G(\mathfrak{g}{}))\\
@V{\textrm{fully faithful}}VV@V{\simeq}V{(\ref{s118e})}V\\
\Rep(\Lie(H)) @>{a^{\vee}}>> \Rep(\mathfrak{g}{})
\end{CD}\quad a\overset{\text{{\tiny def}}}{=}\Lie(b)\circ\eta.
\]

If $a=0$, then $a^{\vee}$ sends all objects to trivial objects, and so the
functor $b^{\vee}$ does the same, which implies that the image of $b$ is $1$.
Hence (\ref{e9}) is injective.

From a homomorphism $a\colon\mathfrak{g}{}\rightarrow\Lie(H)$, we get a tensor
functor%
\[
\Rep(H)\rightarrow\Rep(\Lie(H))\overset{a^{\vee}}{\longrightarrow
}\Rep(\mathfrak{g}{})\simeq\Rep(G{}(\mathfrak{g}{}))
\]
and hence a homomorphism $G(\mathfrak{g}{})\rightarrow H$, which acts as $a$
on the Lie algebras. Hence (\ref{e9}) is surjective.

If $a$ is surjective, then $a^{\vee}$ is fully faithful, and so
$\Rep(H)\rightarrow\Rep($$G(\mathfrak{g}{}))$ is fully faithful, which implies
that $G(\mathfrak{g}{})\rightarrow G$ is surjective.
\end{proof}

\begin{proposition}
\label{s120}For any finite extension $k^{\prime}\supset k$ of fields,
$G(\mathfrak{g}{}_{k^{\prime}})\simeq G(\mathfrak{g}{})_{k^{\prime}}$.
\end{proposition}

\begin{proof}
More precisely, we prove that the pair $(G(\mathfrak{g}{})_{k^{\prime}}%
,\eta_{k^{\prime}})$ obtained from $(G(\mathfrak{g}{}),\eta)$ by extension of
the base field has the universal property characterizing $(G(\mathfrak{g}%
{}_{k^{\prime}}),\eta)$. Let $H$ be an algebraic group over $k^{\prime}$, and
let $H_{\ast}$ be the group over $k$ obtained from $H$ by restriction of the
base field. Then%
\begin{align*}
\Hom_{k^{\prime}}(\mathcal{G}{}(\mathfrak{g}{})_{k^{\prime}},H)  &
\simeq\Hom_{k}(\mathcal{G}{}(\mathfrak{g}{}),H_{\ast})\quad\text{(universal
property of }H_{\ast}\text{)}\\
&  \simeq\Hom_{k}(\mathfrak{g}{},\Lie(H_{\ast}))\quad\text{(\ref{s119})}\\
&  \simeq\Hom_{k^{\prime}}(\mathfrak{g}{}_{k^{\prime}},\Lie(H)).
\end{align*}
For the last isomorphism, note that%
\[
\Lie(H_{\ast})\overset{\text{{\tiny def}}}{=}\Ker(H_{\ast}(k[\varepsilon
])\rightarrow H_{\ast}(k))\simeq\Ker(H(k^{\prime}[\varepsilon])\rightarrow
H(k^{\prime}))\overset{\text{{\tiny def}}}{=}\Lie(H).
\]
In other words, $\Lie(H_{\ast})$ is $\Lie(H)$ regarded as a Lie algebra over
$k$ (instead of $k^{\prime}$), and the isomorphism is simply the canonical
isomorphism in linear algebra,%
\[
\Hom_{k\text{-linear}}(V,W)\simeq\Hom_{k^{\prime}\text{-linear}}(V\otimes
_{k}k^{\prime},W)
\]
($V,W$ vector spaces over $k$ and $k^{\prime}$ respectively).
\end{proof}

The next theorem shows that, when $\mathfrak{g}{}$ is semisimple,
$G(\mathfrak{g}{})$ is a semisimple algebraic group with Lie algebra
$\mathfrak{g}{}$, and any other semisimple group with Lie algebra
$\mathfrak{g}{}$ is a quotient of $G(\mathfrak{g}{})$; moreover, the centre of
$G(\mathfrak{g}{})$ has character group $P/Q$.

\begin{theorem}
\label{s122}Let $\mathfrak{g}{}$ be a semisimple Lie algebra.

\begin{enumerate}
\item The homomorphism $\eta\colon\mathfrak{g}{}\rightarrow\Lie($%
$G(\mathfrak{g}{}))$ is an isomorphism.

\item The group $G(\mathfrak{g}{})$ is a connected semisimple group.

\item For any algebraic group $H$ and isomorphism $a\colon\mathfrak{g}%
{}\rightarrow\Lie(H)\mathfrak{{}}$, there exists a unique isogeny $b\colon
$$G(\mathfrak{g}{})\rightarrow H^{\circ}$ such that $a=\Lie(b)\circ\eta$.

\item Let $Z$ be the centre of $G(\mathfrak{g}{})$; then $X^{\ast}(Z)\simeq
P/Q$.
\end{enumerate}
\end{theorem}

\begin{proof}
(a) Because $\Rep($$G(\mathfrak{g}{}))$ is semisimple, $G(\mathfrak{g}{})$ is
reductive. Therefore $\Lie($$G(\mathfrak{g}{}))$ is reductive (\ref{s110}),
and so $\Lie($$G(\mathfrak{g}{}))=\eta(\mathfrak{g}{})\oplus\mathfrak{a}%
{}\oplus\mathfrak{c}{}$ with $\mathfrak{a}{}$ is semisimple and $\mathfrak{c}%
{}$ commutative. If $\mathfrak{a}{}$ or $\mathfrak{c}{}$ is nonzero, then
there exists a nontrivial representation $r$ of $G(\mathfrak{g}{})$ such that
$\Lie(r)$ is trivial on $\mathfrak{g}{}$. But this is impossible because
$\eta$ defines an equivalence $\Rep($$G(\mathfrak{g}{}))\rightarrow
\Rep(\mathfrak{g}{})$.

(b) Now (\ref{s110}) shows that $G$ is semisimple.

(c) Proposition \ref{s119} shows that there exists a unique homomorphism $b$
such that $a=\Lie(b)\circ\eta$, which is an isogeny because $\Lie(b)$ is an isomorphism.

(d) In the next subsection, we show that if $\mathfrak{g}{}$ is splittable,
then $X^{\ast}(Z)\simeq P/Q$ (as abelian groups). As $\mathfrak{g}{}$ becomes
splittable over a finite Galois extension, this implies (d).
\end{proof}

\begin{remark}
\label{s122r}The isomorphism $X^{\ast}(Z)\simeq P/Q$ in (d) commutes with the
natural actions of $\Gal(k^{\mathrm{al}}/k)$.
\end{remark}

\section{Split semisimple algebraic groups}

Let $(\mathfrak{g}{},\mathfrak{h)}$ be a split semisimple Lie algebra, and let
$P$ and $Q$ be the corresponding weight and root lattices. The action of
$\mathfrak{h}{}$ on a $\mathfrak{g}{}$-module $V$ decomposes it into a direct
sum $V=\bigoplus_{\varpi\in P}V_{\varpi}$. Let $D(P)$ be the diagonalizable
group attached to $P$. Then $\Rep(D(P))$ has a natural identification with the
category of $P$-graded vector spaces. The functor $(V,r_{V})\mapsto
(V,(V_{\varpi})_{\omega\in P})$ is an exact tensor functor $\Rep(\mathfrak{g}%
{})\rightarrow\Rep(D(P))$, and hence defines a homomorphism $D(P)\rightarrow
$$G(\mathfrak{g}{})$. Let $T{}(\mathfrak{h)}$ be the image of this homomorphism.

\begin{theorem}
\label{s123}With the above notations:

\begin{enumerate}
\item The group $T{}(\mathfrak{h}{})$ is a split maximal torus in
$G(\mathfrak{g}{})$, and $\eta$ restricts to an isomorphism $\mathfrak{h}%
{}\rightarrow\Lie(T(\mathfrak{h}{}))$.

\item The map $D(P)\rightarrow T{}(\mathfrak{h}{})$ is an isomorphism;
therefore, $X^{\ast}(T(\mathfrak{h}{}))\simeq P$.

\item The centre of $G(\mathfrak{g}{})$ is contained in $T{}(\mathfrak{h}{})$
and equals%
\[
\bigcap\nolimits_{\alpha\in R}\Ker(\alpha\colon T(\mathfrak{h}{}%
)\rightarrow\mathbb{G}_{m})
\]
(and so has character group $P/Q$).
\end{enumerate}
\end{theorem}

\begin{proof}
(a) The torus $T(\mathfrak{h}{})$ is split because it is the quotient of a
split torus. Certainly, $\eta$ restricts to an injective homomorphism
$\mathfrak{h}{}\rightarrow\Lie(T(\mathfrak{h}{}))$. It must be surjective
because otherwise $\mathfrak{h}{}$ wouldn't be a Cartan subalgebra of
$\mathfrak{g}{}$. The torus $T(\mathfrak{h}{})$ must be maximal because
otherwise $\mathfrak{h}$ wouldn't be equal to its normalizer.

(b) Let $V$ be the representation $\bigoplus V_{\varpi}$ of $\mathfrak{g}{}$
where $\varpi$ runs through a set of fundamental weights. Then $G(\mathfrak{g}%
{})$ acts on $V$, and the map $D(P)\rightarrow\GL(V)$ is injective. Therefore,
$D(P)\rightarrow T(\mathfrak{h})$ is injective.

(c) A gradation on $\Rep(\mathfrak{g}{})$ is defined by a homomorphism
$P\rightarrow M(P_{++})$ (see \ref{s106}), and hence by a homomorphism
$D(M(P_{++}))\rightarrow T(\mathfrak{h}{})$. This shows that the centre of $G$
is contained in $T(\mathfrak{h}{})$. Because the centre of $\mathfrak{g}{}$ is
trivial, the kernel of the adjoint map $\Ad\colon G\rightarrow
\GL_{\mathfrak{g}{}}$ is the centre $Z(G)$ of $G$, and so the kernel of
$\Ad|T(\mathfrak{h})$ is $Z(G)\cap T(\mathfrak{h}{})=Z(G)$. But%
\[
\Ker(\Ad|T(\mathfrak{\mathfrak{h}{}}))=\bigcap_{\alpha\in R}\Ker(\alpha),
\]
so $Z(G)$ is as described.
\end{proof}

\begin{theorem}
\label{s123a}Let $T$ and $T^{\prime}$ be split maximal tori in $G(\mathfrak{g}%
{})$. Then $T^{\prime}=gTg^{-1}$ for some $g\in G(\mathfrak{g}{})(k).$
\end{theorem}

\begin{proof}
Let $x$ be a nilpotent element of $\mathfrak{g}{}$. For any representation
$(V,r_{V})$ of $\mathfrak{g}{}$, $e^{r_{V}(x)}\in G(\mathfrak{g})(k)$. There
exist nilpotent elements $x_{1},\ldots,x_{m}$ in $\mathfrak{g}{}$ such that%
\[
e^{\ad(x_{1})}\cdots e^{\ad(x_{m})}\Lie(T)=\Lie(T^{\prime}).
\]
Let $g=e^{\ad(x_{1})}\cdots e^{\ad(x_{m})}$; then $gTg^{-1}=T^{\prime}$
because they have the same Lie algebra.
\end{proof}

\section{Classification}

We can now read off the classification theorems for split semisimple algebraic
groups from the similar theorems for split semisimple Lie algebras.

Let $(G,T)$ be a split semisimple algebraic group. Because $T$ is
diagonalizable, the $k$-vector space $\mathfrak{g}{}$ decomposes into
eigenspaces under its action:%
\[
\mathfrak{g}{}=\bigoplus_{\alpha\in X^{\ast}(T)}\mathfrak{g}^{\alpha}.
\]
The roots of $(G,T)$ are the nonzero $\alpha$ such that $\mathfrak{g}^{\alpha
}\neq0$. Let $R$ be the set of roots of $(G,T)$.

\begin{proposition}
\label{s124}The set of roots of $(G,T)$ is a reduced root system $R$ in
$V\overset{\text{{\tiny def}}}{=}X^{\ast}(T)\otimes\mathbb{Q}{}$; moreover,%
\begin{equation}
Q(R)\subset X^{\ast}(T)\subset P(R). \label{e13}%
\end{equation}

\end{proposition}

\begin{proof}
Let $\mathfrak{g}{}=\Lie G$ and $\mathfrak{h}{}=\Lie T$. Then $(\mathfrak{g}%
{},\mathfrak{h}{})$ is a split semisimple Lie algebra, and, when we identify
$V$ with a subspace of $\mathfrak{h}{}^{\vee}\simeq X^{\ast}(T)\otimes k$, the
roots of $(G,T)$ coincide with the roots of $(\mathfrak{g}{},\mathfrak{h}{}$)
and (\ref{e13}) holds.
\end{proof}

By a \emph{diagram} $(V,R,X)$, we mean a reduced root system $(V,R)$ over
$\mathbb{Q}{}$ and a lattice $X$ in $V$ that is contained between $Q(R)$ and
$P(R)$.

\begin{theorem}
[Existence]\label{s125}Every diagram arises from a split semisimple algebraic
group over $k$.
\end{theorem}

More precisely, we have the following result.

\begin{theorem}
\label{s126}Let $(V,R,X)$ be a diagram, and let $(\mathfrak{g}{}%
,\mathfrak{h}{})$ be a split semisimple Lie algebra over $k$ with root system
$(V\otimes k,X)$. Let $\Rep(\mathfrak{g}{})^{X}$ be the full subcategory of
$\Rep(\mathfrak{g}{})$ whose objects are those whose simple components have
heighest weight in $X$. Then $\Rep(\mathfrak{g}{})^{X}$ is a tannakian
subcategory of $\Rep(\mathfrak{g}{})$, and there is a natural functor
$\Rep(\mathfrak{g}{})^{X}\rightarrow\Rep(D(X))$. The Tannaka dual $(G,T)$ of
this functor is a split semisimple algebraic group with diagram $(V,R,X)$.
\end{theorem}

\begin{proof}
When $X=Q$, $(G,T)=(G(\mathfrak{g}{}),T(\mathfrak{\mathfrak{h}{}}{}))$, and
the statement follows from Theorem \ref{s123}. For an arbitrary $X$, let%
\[
N=\bigcap\nolimits_{\chi\in X/Q}\Ker(\chi\colon Z(G(\mathfrak{g}%
{}))\rightarrow\mathbb{G}_{m}).
\]
Then $\Rep(\mathfrak{g}{})^{X}$ is the subcategory of $\Rep(\mathfrak{g}{})$
on which $N$ acts trivially, and so it is a tannakian category with Tannaka
dual $G(\mathfrak{g}{})/N$. Now it is clear that $(G(\mathfrak{g}%
{})/N,T(\mathfrak{h})/N)$ is the Tannaka dual of $\Rep(\mathfrak{g}{}%
)^{X}\rightarrow\Rep(D(X))$, and that it has diagram $(V,R,X)$.
\end{proof}

\begin{theorem}
[Isogeny]\label{s127}Let $(G,T)$ and $(G^{\prime},T^{\prime})$ be split
semisimple algebraic groups over $k$, and let $(V,R,X)$ and $(V,R^{\prime
},X^{\prime})$ be their associated diagrams. Any isomorphism $V\rightarrow
V^{\prime}$ sending $R\ $onto $R^{\prime}$ and $X$ into $X^{\prime}$ arises
from an isogeny $G\rightarrow G^{\prime}$ mapping $T$ onto $T^{\prime}$.
\end{theorem}

\begin{proof}
Let $(\mathfrak{g}{},\mathfrak{h})$ and $(\mathfrak{g}{}^{\prime}%
,\mathfrak{h}{}^{\prime})$ be the split semisimple Lie algebras of $(G,T)$ and
$(G^{\prime},T^{\prime})$. An isomorphism $V\rightarrow V^{\prime}$ sending
$R$ onto $R^{\prime}$ and $X$ into $X^{\prime}$ arises from an isomorphism
$(\mathfrak{g}{},\mathfrak{h)}\overset{\beta}{\longrightarrow}(\mathfrak{g}%
{}^{\prime},\mathfrak{h}{}^{\prime})$. Now $\beta$ defines an exact tensor
functor $\Rep(\mathfrak{g}{}^{\prime})^{X^{\prime}}\rightarrow
\Rep(\mathfrak{g}{})^{X}$, and hence a homomorphism $\alpha\colon G\rightarrow
G^{\prime}$, which has the required properties.
\end{proof}

\bibliographystyle{cbe}
\bibliography{../../../refs}

\end{document}